\documentclass[11pt,reqno,sort]{elsarticle}

\usepackage{amsmath,amssymb,amsfonts,amsthm,latexsym,graphicx,multirow,color,hyperref,enumerate, footnote}
\usepackage[all]{xy}

\usepackage{caption,stmaryrd}
\def\Box{\vcenter{\vbox{\hrule\hbox{\vrule
     \vbox to 8.8pt{\hbox to 10pt{}\vfill}\vrule}\hrule}}}

\newcommand{\eproof}{\hfill$\Box$\vspace{4mm}}

\newcommand{\PSL}{\mathrm{PSL}}

\newtheorem{thm}{Theorem}[section]
\newtheorem{lemma}[thm]{Lemma}
\newtheorem{cor}[thm]{Corollary}
\newtheorem{prop}[thm]{Proposition}

\numberwithin{equation}{section}
\newtheorem{remark}[thm]{Remark}

\newcommand{\cL}{\mathcal L}

\newcommand{\cP}{\mathcal P}
\newcommand{\cS}{\mathcal S}

\definecolor{Purple}{rgb}{0.5,0,0.5}

\def\Aut{{\rm Aut}}

 \def\PSL{{\rm PSL}}

\def\N{{\rm N}}
\def\C{{\rm C}}
\def\soc{{\rm soc}}

\begin{document}
\newcommand{\stopthm}{\begin{flushright}
		\(\box \;\;\;\;\;\;\;\;\;\; \)
\end{flushright}}
\newcommand{\symfont}{\fam \mathfam}

\title{On the multipliers of a Singer quadrangle}

\author[add1]{Wendi Di}\ead{diwd@ecust.edu.cn}
\author[add2]{Tao Feng \corref{cor1}}\ead{tfeng@zju.edu.cn}\cortext[cor1]{Corresponding author}
\address[add1]{School of Mathematics, East China  University of Science and Technology, Shanghai 200237, China}
\address[add2]{School of Mathematical Sciences, Zhejiang University, Hangzhou 310058,  China}

\begin{abstract}
A finite generalized quadrangle $\cS$ is a Singer quadrangle if it has an automorphism group that acts sharply transitively on its points. In this paper, we introduce the notion of multipliers for a Singer quadrangle and study their basic properties. As an application, we show that a point-primitive automorphism group of a thick generalized quadrangle cannot have O'Nan-Scott type HS (holomorph simple), which answers an open problem in \cite{Bamberg 2019}.
\newline
\noindent\text{Keywords:} 
generalized quadrangle, Singer quadrangle,  multiplier, primitive group
		
\end{abstract}

\maketitle

\section{Introduction}\label{Sec_int}
Let $\cS$ be a point-line incidence structure with point set $\cP$ and line set $\cL$, and assume that $\cP,\cL$ are both finite. It is a generalized $n$-gon if its incidence graph is a bipartite graph of diameter $n$ and girth $2n$. We say that a generalized $n$-gon $\cS$ has order $(s,t)$ if each point is incident with $t+1$ lines and each line is incident with $s+1$ points, and we say that it is thick  if both $s$ and $t$ are at least $2$. A finite thick generalized $3$-gon is simply a finite projective plane.

A Singer group of a generalized $n$-gon $\cS$ is an automorphism group of $\cS$ that acts sharply transitively on its points. For properties of Singer groups of a finite projective plane and their applications, we refer to \cite{Ho93basic,HoPAMS}. For brevity, we refer to a generalized quadrangle with a Singer group as a Singer quadrangle, following \cite{WTSSinger}. Ghinelli initiated the study of Singer quadrangles in \cite{Ghinelli}. She conjectured in \cite{Ghinelli} that a Singer quadrangle of even order does not exist, which was proved independently in \cite{Feng} and \cite{Ott}. Yoshiara \cite{Yoshiara} developed several fundamental lemmas on Singer quadrangles and showed that they cannot have order $(t^2,t)$. De Winter and K. Thas gave a satisfactory characterization of the generalized quadrangles with a Singer group that is abelian or an odd order Heisenberg group of dimension $3$ in \cite{Winter&Thas} and \cite{Winter&Thas2} respectively. The Singer group of a generalized quadrangle can have complicated structures and its nilpotency class can be arbitrarily large, cf. \cite{Bamberg2011,FengLi}. It seems intractable to determine all such groups that can act regularly on the points of a generalized quadrangle at this moment. In \cite{Swartz}, Swartz studied the Singer quadrangles of order $(u^2,u^3)$.

The concept of a multiplier for finite projective planes was introduced in \cite{Hall47}, and we refer to Ho's survey \cite{Ho93}. It has a natural generalization to generalized polygons as follows. Fix a point $p$. Let $G$ be a Singer group of a generalized $n$-gon $\cS$, and we identify the point set $\cP$ with $G$ via $p^g\mapsto g$ for $g\in G$. We say that an automorphism $\theta$ of $G$ is a multiplier of $\cS$ if it induces an automorphism of the geometry $\cS$.  In this paper, we conduct the first systematic study of  the properties of multipliers of a Singer quadrangle.

The study of highly transitive generalized polygons is an active area of research, cf. \cite{Bamberg2012,Bamberg2016,Bamberg Popiel2017,Bamberg2021,Lu,Zou} for some recent papers. The classification of finite simple groups and the O'Nan-Scott Theorem for primitive permutation groups play a fundamental role, and we take the notation for the O'Nan-Scott types as introduced in \cite{Praeger}. Bamberg et al. studied the point-primitive generalized quadrangles in \cite{Bamberg 2019}, and obtained restrictions for each nonaffine O'Nan-Scott type. They showed that the HC (holomorph compound) case does not occur, and showed confidence that the HS (holomorph simple) case is within reach. As pointed out in \cite[Section 8.1]{Bamberg 2019}, the generalized quadrangle has a regular automorphism group in the HS, SD (simple diagonal), CD (compound diagonal) and TW (twisted wreath) cases respectively. As an application of our results on multipliers, we are able to exclude the O'Nan-Scott type HS as expected and extend the results for the SD, CD cases in \cite{Bamberg 2019}. We summarize the results in the next theorem.
\begin{thm}\label{thm_PtPrim}
Let $G$ be an automorphism group of a finite  thick generalized quadrangle $\cS=(\cP,\cL)$ which is primitive on  $\cP$. Then the action of $G$ on $\cP$ does not have O'Nan-Scott type HS or HC, and the conditions in Table~\ref{tab:1} hold for the O'Nan-Scott type SD and CD.
\end{thm}
 \begin{table}[h!]
\centering
\caption{}
\begin{tabular}{ccc}
\hline
        Type  &$\soc(G)$ & Possibilities on $T$  \\
\hline
        CD &$(T^k)^2$, $k\ge 2$ & $T\cong \rm{Alt}(n)$ with $ n\le 7$\\
 && $T\cong \rm{M}_{11}$, $\rm{M}_{12}$, $\rm{M}_{22}$, $\rm{M}_{23}$, Th, ${\rm J}_1$, O'N, $\rm{J}_3$, Ru, Ly\\
&& $T$ is of  Lie type ${}^2{\rm G}_2(q)$, ${}^2{\rm B}_2(q)$, $\PSL(2,q)$, $\PSL(3,q)$, $\rm{PSU}(3,q)$\\

\hline
SD& $T^k$, $k\geq 3$& $T\cong \rm{Alt}(n)$ with $ n\le 15$\\
&& $T$ is a sporadic simple group \\
 && $T$ is an exceptional  simple group with $T\ncong \rm{E}_8$\\
&& $T$ is of Lie type $\PSL(n,q)$ with $2\le n\le 7$, ${\rm PSU}(n,q)$ with $3\le n\le 6$, \\
&& ${\rm PSp}(2n,q)$ with $2\le n\le  3$, ${\rm \Omega}(7,q)$ or ${\rm P\Omega^\epsilon}(2n,q)$ with $4\le n\le 8$\\
\hline
\end{tabular}
\label{tab:1}
\end{table}

This paper is organized as follows. In Section 2, we present some preliminary facts on finite generalized quadrangles. In Section 3, we establish some basic properties of a multiplier of a Singer quadrangle. In Section 4, we show that a point-primitive automorphism group of a thick generalized quadrangle cannot have O'Nan-Scott type HS by using the results in Section 3. In Section 5, we complete the proof of Theorem \ref{thm_PtPrim} by excluding some potential point-primitive groups of type SD and CD listed in \cite[Theorem 1.1]{Bamberg 2019}.

\section{Preliminaries}\label{Sec_pre}
A \emph{grid} with parameters $(s_{1},s_{2})$ is a point-line incidence structure $(\mathcal{P}, \mathcal{L}, \mathrm{I})$ with point set $\mathcal{P}=\left\{x_{i, j}: 0 \le i\le s_1, 0 \le j\le s_2\right\}$ and line set $\mathcal{L}=\left\{\ell_{0}, \ldots, \ell_{s_{1}}, \ell_{0}^{\prime}, \ldots, \ell_{s_{2}}^{\prime}\right\}$ such that $x_{i, j} \mathrm{I} \ell_{k}$ if and only if $i=k$ and $x_{i, j} \mathrm{I} \ell_{k}^{\prime}$ if and only if $j=k$. A \emph{dual grid} with parameters $(s_1,s_2)$ is the point-line dual of a grid with parameters $(s_2,s_1)$. A grid with parameters $(s,s)$ is a thin generalized quadrangle of order $(s,1)$, and a dual grid with parameters $(t,t)$ is a thin generalized quadrangle of order $(1,t)$.\medskip

Let $\cS=(\cP,\cL)$ be a finite thick generalized quadrangle of order $(s,t)$, where $\cP$ is the point set and $\cL$ is the line set. For two distinct points $p,q$, we write $p\sim q$ if either $p=q$ or they are collinear. We define the relation $\sim$ dually for lines. For an automorphism $g$ of $\cS$, we define $\mathcal{P}_0(g)=\{p\in \mathcal{P}\mid p^g=p\}$ and
\[
  \mathcal{P}_1(g)=\{p\in \mathcal{P}\mid p^g\ne p,p^g\sim p\},\quad \mathcal{P}_2(g)=\{p\in \mathcal{P}\mid p^g\neq p,p^g\nsim p\}.
\]
They form a partition of $\cP$, and $\cP_0(g)$ is the set of fixed points. We define $\cL_0(g)$, $\cL_1(g)$ and $\cL_2(g)$ dually for lines, which form a partition of $\cL$. Let $\mathcal{S}_{g}=(\cP_0(g),\cL_0(g))$ be the induced incidence substructure on $\mathcal{P}_0(g)\times\mathcal{L}_0(g)$, and we call it the fixed substructure of $g$.\medskip

We state some standard facts from the monograph \cite{Payne} in the next few lemmas.
\begin{lemma}\emph{\cite[1.2.1-1.2.3]{Payne}}\label{lem_HD}
Let $\mathcal{S}$ be a finite thick generalized quadrangle of order $(s,t)$ with point set $\mathcal{P}$ and line set $\mathcal{L}$. Then
\begin{equation}\label{eq_size}
  |\mathcal{P}|=(s+1)(st+1),\,|\mathcal{L}|=(t+1)(st+1),
\end{equation}
 and the following properties hold:
\begin{enumerate}
\item[(i)] (Higman's inequality) $s\leq t^{2}$ and $t\leq s^{2}$;
\item[(ii)] (Divisibility condition) $s+t$ divides $st(s+1)(t+1)$.
\end{enumerate}
\end{lemma}

\begin{lemma}\emph{\cite[2.4.1]{Payne}}\label{subquadrangle}
Let $g$ be an automorphism of a finite thick generalized quadrangle $\mathcal{S}=(\mathcal{P}, \mathcal{L})$ of order $(s,t)$, and let $\mathcal{S}_{g}=(\cP_0(g),\cL_0(g))$ be the fixed substructure. Then one of the following holds:
\begin{enumerate}
\item[(0)] $\mathcal{P}_0(g)=\mathcal{L}_0(g)=\varnothing$;
\item[(1)] $\mathcal{L}_0(g)=\varnothing$, $\mathcal{P}_0(g)$ is a nonempty set of pairwise noncollinear points;
\item[(1')] $\mathcal{P}_0(g)=\varnothing$, $\mathcal{L}_0(g)$ is a nonempty set of pairwise nonconcurrent lines;
\item[(2)] $\mathcal{L}_0(g)$ is nonempty, and $\mathcal{P}_0(g)$ contains a point $P$ that is collinear with each other point of $\mathcal{P}_0(g)$ and is on each line of $\mathcal{L}_0(g)$;
\item[(2')]$\mathcal{P}_0(g)$ is nonempty, and $\mathcal{L}_0(g)$ contains a line $\ell$ that is concurrent with each other  line of $\mathcal{L}_0(g)$ and contains each point of $\mathcal{P}_0(g)$;
\item[(3)] $\mathcal{S}_{g}$ is a grid with parameters $(s_1,s_2)$, $s_{1}\le s_{2}$;
\item[(3')] $\mathcal{S}_{g}$ is a dual grid with parameters $(t_1,t_2)$, $t_{1}\le t_{2}$;
\item[(4)] $\mathcal{S}_{g}$ is a finite generalized quadrangle of order $\left(s',t'\right)$ with $\min\{s',t'\}\ge 2$.
\end{enumerate}
\end{lemma}

\begin{lemma}\emph{\cite[1.9.1, 1.9.2]{Payne}}\label{lem_bensonlem}
  Let $\mathcal{S}$ be a generalized quadrangle of order $(s,t)$ and let $g$ be an automorphism of $\mathcal{S}$.
   Then we have
    \begin{equation}\label{eq_Bensonlem}
    (1+t)|\mathcal{P}_0(g)|+|\mathcal{P}_1(g)|=(1+s)|\mathcal{L}_0(g)|+|\mathcal{L}_1(g)|\equiv (s+1)(t+1) \pmod {s+t}.
  \end{equation}
\end{lemma}
\begin{lemma}\emph{\cite[2.2.1]{Payne}}\label{lem_subgq1}
  Let $\mathcal{S'}$ be a proper subquadrangle of $\cS$ with order $(s',t')$. Then either $s=s'$ or $s't'\le s$. Dually, we have either $t=t'$ or $s't'\le t$.
\end{lemma}

\begin{remark}\label{rem_smallpara}
When $\min\{s,t\}\le 3$, all the finite generalized quadrangles of order $(s,t)$ are known, cf. {\rm\cite[Chapter 6]{Payne}}. Up to duality, such a quadrangle is either one of the  classical  quadrangles $W(2)$, $W(3)$, $\mathcal{Q}(4,3)$, $\mathcal{Q}^-(5,2)$ and $\mathcal{Q}^-(5,3)$, or is the Payne derived quadrangle of $W(4)$. Among those quadrangles and their duals, there are exactly two Singer quadrangles: $\mathcal{Q}^-(5,2)$ and the Payne derived quadrangle of $W(4)$. For $\mathcal{Q}^-(5,2)$, the full automorphism group acts primitively on both points and lines (cf. {\em \cite[Theorem 1.1]{Bamberg2011})};  for the Payne derived quadrangle of $W(4)$, the full automorphism group  is primitive only on points (cf. {\em \cite[Corollary 1.5]{Bamberg2016}}).
\end{remark}

\section{The multiplier group of a Singer quadrangle}

Let  $\cS=(\cP,\cL)$ be a thick generalized quadrangle of order $(s,t)$, and suppose that it has an automorphism group $G$ that is regular on the point set $\cP$. Fix a point $p$, and we identify $\cP$ with $G$ via $p^g\mapsto g$, $g\in G$. The group $G$ acts on $\cP$ via right multiplication. Let
\begin{equation}\label{eqn_Deltadef}
  \Delta=\{g\in G\setminus\{1\}:1\sim g\}.
\end{equation}
This is slightly different from the definition in \cite{Yoshiara}, since we assume $1\not\in\Delta$ here. For a line $\ell$, we write $\bar{\ell}$ for the set of points collinear with $\ell$.\medskip

An automorphism $\theta$ of $G$ is a multiplier of $\cS$ if it induces an automorphism of the geometry $\cS$. The set of all multipliers of $\cS$ forms a subgroup of $\Aut(G)$, which we call the multiplier group of $\cS$ with respect to the Singer group $G$. For $\theta\in\Aut(G)$, let $C_G(\theta)=\{g\in G\mid g^\theta=g\}$.
\begin{prop}\label{prop_FixMult}
For a multiplier $\theta$,  let $X=\{g^{\theta}g^{-1}\mid g\in G\}$.
\begin{itemize}
  \item[(1)]We have $\cP_0(\theta)=C_G(\theta)$, $\cP_1(\theta)=\{g\in G:g^{\theta}g^{-1}\in\Delta\}$, and $|\cP_1(\theta)|=|C_G(\theta)|\cdot |X\cap\Delta|$.
  \item[(2)]There is a constant $c$ such that there are $c$ fixed lines incident with each point of $\cP_0(\theta)$.
\end{itemize}
\end{prop}
\begin{proof}
(1) We write $H=C_G(\theta)$ in this proof. It is clear that $\cP_0(\theta)=H$. Take $g\in G\setminus H$. We have $g\in\cP_1(\theta)$, i.e., $g^\theta\sim g$, if and only if  $g^\theta g^{-1}\in\Delta$. Therefore, we have $\cP_1(\theta)=\{g\in G:g^{\theta}g^{-1}\in\Delta\}$. For $g,h\in G$ such that $g^{\theta}g^{-1}=h^{\theta}h^{-1}$, we have $\left(g^{-1}h\right)^\theta=g^{-1}h$, i.e., $g^{-1}h\in H$. It follows that $|\cP_1(\theta)|=|H|\cdot |X\cap\Delta|$.

(2) For a fixed line $\ell$ incident with $1$ and an element $h\in H$, $\ell^h$ is a fixed line incident with the point $h$. This completes the proof.
\end{proof}

\begin{prop}\label{prop_FixMult2}
Take notation as in Proposition \ref{prop_FixMult}, and assume that $o(\theta)=2$ or $3$.
\begin{itemize}
  \item[(1)] The group $C_G(\theta)$ acts semiregularly on $\cL_1(\theta)$.
  \item[(2)] We have $|\mathcal{L}_0(\theta)|=\frac{|H|}{1+s}(c+|X\cap\Delta|)$,  $|\cL_1(\theta)|=(t+1-c)|C_G(\theta)|$.
\end{itemize}
\end{prop}
\begin{proof}
Let $H=C_G(\theta)$. Take $\ell\in\cL_1(\theta)$. If $o(\theta)=2$, there is a unique fixed point on $\ell$, i.e., $x=\ell\cap\ell^\theta$. If $o(\theta)=3$, $\ell,\ell^\theta,\ell^{\theta^2}$ are three distinct pairwise concurrent lines, so they must be incident with the same point $x$. In both cases, $\ell$ contains a unique fixed point $x$. If $\ell^h=\ell$ for some $h\in H$, then $xh$ is also a fixed point by $\theta$ on $\ell$, yielding that $h=1$. Hence $H$ is semiregular on the set $\cL_1(\theta)$. This proves (1).

Consider the induced subgraph $\Gamma'$ of the incidence graph of $\cS$ by the vertex set $\cL_1(\theta)\cup\cP_0(\theta)$, i.e., the bipartite graph such that $q\in\cP_0(\theta)$ is adjacent to  $\ell\in\cL_1(\theta)$ if and only if $q$ is incident with $\ell$. By Proposition \ref{prop_FixMult} (2), each vertex $q$ in $\cP_0(\theta)$ has $t+1-c$ neighbors in $\Gamma'$. By double counting the number of edges in $\Gamma'$, we deduce that $|\cL_1(\theta)|=(t+1-c)|H|$. By Lemma \ref{lem_bensonlem} and Proposition \ref{prop_FixMult}, we deduce that
$|H|(1+t+|X\cap\Delta|)=(1+s)|\mathcal{L}_0(\theta)|+(t+1-c)|H|$. This gives the desired expression for $|\mathcal{L}_0(\theta)|$. This proves (2) and completes the proof.
\end{proof}

\begin{thm}\label{thm:C_T(u)}
Suppose that $\mathcal{S}=(\mathcal{P},\mathcal{L})$ is a finite  thick generalized quadrangle of order $(s,t)$ admitting a point-regular automorphism group $G$. Let $\theta\in\Aut(G)$ be a multiplier, and write $H=C_G(\theta)$. Then either $H=1$, or $H>1$ and the fixed substructure $\cS_\theta=(\cP_0(\theta),\cL_0(\theta))$ is one of the following cases:
\begin{enumerate}
  \item [(a)] $\cS_\theta$  is a partial ovoid, and  $|H|\le 1+st$.
  \item [(b)] The points in $H$ are on the same line, and $|H|$ divides $1+s$.
  \item [(c)] $\cS_\theta$ is a grid with parameter $(s_1,s_2)$, and there are subgroups $Y_1,Y_2$ of $H$ such that $H=Y_1Y_2$, and $|Y_i|=s_i+1$, $s_i+1\mid s+1$, the points in $Y_i$ lie on a line  for $i=1,2$.
  \item[(c')]$\cS_\theta$ is a thin subquadrangle of order $(s_1,1)$ with $s_1\le \min\{s,t\}$, and $H=\langle X_1,g_1\rangle$ for a subgroup $X_1$ of order $\frac{s_1+1}{2}$, and the two fixed lines incident with $1$ intersect $\Delta$ in $X_1\cup g_1X_1$ and $g_1X_1g_1^{-1}\cup X_1g_1^{-1}$ respectively.
\item [(d)] $\cS_\theta$ is a thin subquadrangle of order $(1,t_1)$ with $2\le t_1\le \min\{s,t\}$, and $H$ has a subgroup of index $2$ consisting of pairwise noncollinear points.
\item [(e)] $\cS_\theta$ is a subquadrangle of order $(s',t')$ with $\min\{s',t'\}\ge 2$.
\end{enumerate}
\end{thm}
\begin{proof}
Suppose that $|H|\ge 2$. There are at least two fixed points, and there is a constant $c$ such that there are $c$ fixed lines incident with each point of $\cP_0(\theta)$ by Proposition \ref{prop_FixMult} (2). This excludes the cases (0), (1') and (2) in Lemma \ref{subquadrangle}. If $\cS_\theta$ is a partial ovoid, then $|\cP_0(\theta)|\le 1+st$ by \cite[1.8.1]{Payne}, and we have case (a). If the points in $H$ lie on a line $\ell$, then $H$ stabilizes $\ell$ and $\bar{\ell}$ is the union of left $H$-cosets. It follows that $|H|$ divides $1+s$, and we have case (b). If $\cS_\theta$ is a subquadrangle, then we have case (e).


Assume that $\cS_\theta$ is a dual grid with parameters $(t_1,t_2)$, which is case (3') in Lemma \ref{subquadrangle}. There are disjoint sets $X_1$, $X_2$ of pairwise noncollinear points such that $|X_i|=t_i+1$ for $i\in\{1,2\}$ and $X_1\subseteq X_2^\perp$, $H= X_1\cup X_2$, and $\cL_0(\theta)$ consists of all the lines that are incident with a point of $X_1$ and a point of $X_2$ respectively. We assume without loss of generality that $1\in X_1$. We have $c=t_1+1=t_2+1$, so $t_1=t_2$. By \cite[1.4.1]{Payne}, we deduce that $t_1\le s$. It is clear that $t_1\le t$, so $t_1\le\min\{s,t\}$.  For $x_1\in X_1$ and $x_2\in X_2$, we deduce from $1\sim x_2$ that $x_1\sim x_2x_1$. It follows that $x_2x_1\in x_1^\perp\cap H=X_2$, so $X_2x_1\subseteq X_2$. We deduce that $X_2$ is the union of left cosets of $\langle X_1\rangle$, where $\langle X_1\rangle$ is the subgroup generated by $X_1$. By the fact $|X_1|=|X_2|$, we deduce that $X_1$ is a subgroup and $X_2$ is a coset of $X_1$. This yields case (d).

It remains to consider the case where $\cS_\theta$ is a grid with parameter $(s_1,s_2)$. Let $\ell_1,\ell_2$ be the two fixed lines incident with $1$, and define $Y_i=\overline{\ell_i}\cap H$ for $i=1,2$. We assume without loss of generality that $|Y_i|=s_i+1$ for $i=1,2$. We have $H\cap\Delta=Y_1\cup Y_2$. For $g\in Y_1$, the line $\ell_1^{g^{-1}}$ is a fixed line through $1$, so it is either $\ell_1$ or $\ell_2$. We consider two separate cases:
\begin{itemize}
  \item[(A)]$\ell_1^{g^{-1}}=\ell_1$ for each $g\in Y_1$,
  \item[(B)]$\ell_2^{g^{-1}}=\ell_2$ for some $g\in Y_1$.
\end{itemize}
In particular, if $s_1\ne s_2$, then only case (A) occurs. Suppose that (A) occurs. In this case, we deduce that $Y_1$ is the stabilizer of $\ell_1$ in $H$. It follows that $Y_1$ is a subgroup and $\overline{\ell_1}$ is the union of left $Y_1$-cosets, so $s_1+1\mid s+1$. Similarly, we deduce that $Y_2$ is a subgroup and $s_2+1\mid s+1$. We have $Y_1\cap Y_2=1$ by the fact $\overline{\ell_1}\cap\overline{\ell_2}=1$, so $H=Y_1Y_2$ by comparing sizes. This gives case (c). Suppose that (B) occurs. In this case, we have $s_1=s_2$. By the dual version of \cite[1.4.1]{Payne}, we deduce that $s_1\le \min\{s,t\}$. Let $X_i=\{g\in Y_i:\ell_i^{g^{-1}}=\ell_i\}$ and $Z_i=Y_i\setminus X_i$ for $i=1,2$. As in case (A), $X_i$ is the stabilizer of $\ell_i$ in $H$ and $Z_i$ is the union of left $X_i$-cosets. On the other hand, $\ell_i^{g^{-1}}=\ell_i$ for $g\in Y_i$ if and only if $g^{-1}$ is in $Y_i$, i.e., $X_i=\{g\in Y_i:g^{-1}\in Y_i\}$. We deduce that $Z_2=\{g^{-1}:g\in Z_1\}$. For $g\in Z_1$ and $h\in Z_2$, we have $\ell_1^{g^{-1}h^{-1}}=\ell_1$, so $g^{-1}h^{-1}\in X_1$. It follows that $\{g^{-1}h^{-1}:g\in Z_1,h\in Z_2\}\subseteq X_1$. By comparing sizes, we deduce that $Z_1=g_1X_1$ for some $g_1\in H$, and so $Z_2=X_1g_1^{-1}$. Similarly, we have $\{g^{-1}h^{-1}:g\in Z_2,h\in Z_2\}\subseteq X_2$, and so $X_2=g_1X_1g_1^{-1}$. This completes the proof.
\end{proof}

\begin{cor}\label{cor:2/3}
Let $\mathcal{S}=(\cP,\cL)$ be a generalized quadrangle of order $(s, t)$, where $\min\{s,t\}\ge 4$. If $G$ is a Singer group of $\cS$ and $\theta \in \text{Aut}(G)$ is a multiplier, then $|C_G(\theta)| < |G|^{3/4}$.
\end{cor}
\begin{proof}
We set $H = C_G(\theta)$. The claim holds trivially if $H=1$, so assume that $|H| \ge 2$. By Theorem \ref{thm:C_T(u)}, $\mathcal{S}_\theta$ is one of the cases (a)-(e). For (b), $|H|\le 1+s$ and so $|H|<|G|^{1/2}$. For (d), we have $|H| = 2(1+t_1)$ with $2\le t_1\le \min\{s,t\}$, so $|H|\le 2(1+s) < |G|^{1/2}$ by the assumption $\min\{s,t\}\ge 4$. We examine the remaining cases in the following.

For (a), $\mathcal{S}_\theta$ is a partial ovoid, and $|H| \le 1+st$. We have $t\le s^2$ by Lemma \ref{lem_HD} $(i)$, so $1+st<(1+s)^3$. It follows that
\[
 |H|^4 \le (1+st)^4 < (1+s)^3(1+st)^3 = |G|^3,\quad\textup{i.e., } |H|<|G|^{3/4}.
\]

For (c), $\cS_\theta$ is a grid with parameter $(s_1,s_2)$, where $s_i\le s$ for $i=1,2$. We assume without loss of generality that $s_1\le s_2$, so that $\cS_\theta$ contains a subquadrangle of order $(s_1,1)$. We first consider the case $s_1=s_2=s$. We have $s\le t$ by \cite[2.2.2(i)]{Payne}. It is straightforward to verify that $(1+s)^5<(1+st)^3$ when $\min\{s,t\}\ge 4$. It follows that
\[
  |H|^4 =(1+s)^8 <(1+s)^3(1+st)^3=|G|^3,\quad\textup{i.e., } |H|<|G|^{3/4}.
\]

We next consider the case $s_1<s$. We have $s_1\le \min\{s,t\}$ by Lemma \ref{lem_subgq1}, so
\begin{equation}\label{eqn_tt111}
  |H|\le (1+s)(1+\min\{s,t\}).
\end{equation}
If $t \ge s$, then we deduce that $|H|<|G|^{3/4}$ as in the case $s_1=s_2=s$.  If $t<s$, it holds that $(1+t)^4(1+s)<(1+st)^3$ upon direct check. It follows that
\[
  |H|^4 \le (1+t)^4(1+s)^4 <(1+s)^3(1+st)^3=|G|^3,
\]
i.e., $|H|<|G|^{3/4}$. The case (c') is handled similarly as in the case (c).

It remains to consider (e), where $\mathcal{S}_\theta$ is a subquadrangle of order $(s', t')$ with $\min\{s',t'\}\ge 2$. First assume that $s'<s$ and $t'<t$, so that $s't'\le \textup{min}\{s,t\}$ by Lemma \ref{lem_subgq1}. If $s< t$, we have $s'\le s^{2/3}$ by the facts $s'^{1/2}\le t'$ and $s'^{3/2}\le s't'\le s$. It holds that
\[
  |H|=(1+s')(1+s't')\le (1+s^{2/3})(1+s)\le(1+s)^{3/4}(1+s^2)^{3/4}<|G|^{3/4},
\]
where the second inequality holds for all $s\ge 3$. If $s\ge t$, then $s't'\le t$, and it is routine to check that $(1+s)(1+t)^4<(1+st)^3$ holds for all $(s,t)$ pairs such that $4\le t\le s\le t^2$. It follows that
\[
 |H|^4\le (1+s)^4(1+t)^4\le(1+s)^{3}(1+st)^{3}=|G|^{3},\textup{ i.e., }|H|<|G|^{3/4}.
\]
We next consider the case $t=t'$. We have $s't\le s$ by Lemma  \ref{lem_subgq1}, and so
\[
  |H|^3=(1+s')^3(1+s't)^3\le (1+s/t)^3(1+s)^3\le(1+s)^{2}(1+st)^{2}< |\cP|^{2},
\]
i.e., $|H|<|\cP|^{2/3}$. The second inequality here holds for all pairs $(s,t)$ such that $4\le t\le s\le t^2$, and we omit the details. Finally, we consider the case $s'=s$. We have $st'\le t$ by Lemma  \ref{lem_subgq1}, so $t'\le t/s\le s$.  Since $|H|$ divides $|G|$, we deduce that $1+st=m(1+st')$ for some integer $m>1$. We have $m\equiv 1\pmod{s}$ by taking modulo $s$ on both sides,  so $m=1+ns$ for some $n\in\mathbb{N}$. It follows that $1+st\ge (1+st')(1+s)$. We have $1+st'<(1+s)^2$ by the fact $t'\le s$, and
it is now routine to show that $|H|<|G|^{3/4}$. This completes the proof.
\end{proof}

Corollary \ref{cor:2/3} is an improvement of \cite[Theorem 1.3]{Bamberg 2019} about the size of the set $P_0(\theta)$ of fixed points when the collineation $\theta$ is a multiplier. This is possible because $|P_0(\theta)|=|C_G(\theta)|$ divides $|\cP|=|G|$, where $G$ is the associated Singer group.

\section{Point-primitive automorphism group of  HS type}\label{Sec_proof}

In this section, we establish the following result.
\begin{thm}\label{thm_HSnonext}
If $\cS$ is a finite thick generalized quadrangle and $G$ is point-primitive automorphism group, then $G$ is not of O'Nan-Scott type HS on the point set.
\end{thm}
We fix the following notation throughout this section. Let $\cS=(\cP,\cL)$ be a finite thick generalized quadrangle of order $(s,t)$ and $G$ is an automorphism group that acts primitively on the point set $\cP$ with O'Nan-Scott type HS (holomorphism simple). There is a nonabelian simple group $T$ such that $G$ has socle $T\times T$, and there is a point $p\in\cP$ whose stabilizer is $D=\{(g,g):g\in T\}$. The group $M=\{(g,1):g\in T\}$ acts regularly on $\cP$, and we identify $\cP$ with $T$ via $p^{(g,1)}\mapsto g$ for $g\in T$. The group $T\times T$ then acts on $\cP= T$ as follows:
\[
  y^{(g_1,g_2)}=g_2^{-1}yg_1,\quad \textup{ for }y,g_1,g_2\in T.
\]
In particular, the group $D$ acts on $\cP=T$ as a group of multipliers. For a line $\ell$, we write
\[
  \bar{\ell}=\{g\in T:\,\textup{ the point $g$ is incident with } \ell\}.
\]
We write $\Delta$ for the set of points collinear but not equal to $1$.  By \cite[Lemma 5.3]{Bamberg 2019}, we have the following important properties:
\begin{itemize}
  \item[(B1)]$\Delta$ is a union of conjugacy classes of $T$,
  \item[(B2)]For each line $m$ incident with $1$, $\bar{m}$ is a subgroup of $T$ and it contains an involution.
\end{itemize}
We fix a line $\ell$ incident with $1$, and take an involution $u$ in $\bar{\ell}$. We set $g=(u,u)\in D$, which is a multiplier of $\cS$. It is clear $\cP_0(g)=C_T(u)$.
\begin{lemma}\label{lem_typeHS}
We have $s+t\mid 1+st$, and $s+2\leq t\leq s^2-s$.
\end{lemma}
\begin{proof}
Take a nonidentity element $x\in T$. By \cite[Lemma 3]{Yoshiara}, we have
\[
 |x^T\cap\Delta|\cdot|C_T(x)|\equiv 1+st\pmod{s+t},
\]
where $x^T$ is the conjugacy class of $x$ in $T$. By the property (B1), either $x^T\cap\Delta=\varnothing$ or $x^T$. Both cases occur as $x$ vary, so
\[
  |T|\equiv 0\equiv 1+st\pmod{s+t}
\]
by the fact $|x^T|\cdot|C_T(x)|=|T|$. It follows that $s+t\mid 1+st$. By \cite[Lemma 5.3 (vii)]{Bamberg 2019}, we have $\gcd(s,t)=1$ and $t\ge s+1$. It is elementary to show that $s+t$ does not divide $1+st$ if $t=s+1$ or $t=s^2$.  By \cite[1.2.5]{Payne}, we have $t\leq s^2-s$. This completes the proof.
\end{proof}

\begin{lemma}\label{lem_FixgHS}
Take notation as above, and let $H=C_T(u)$.
\begin{itemize}
  \item[(1)]For $\ell'\in\cL_0(g)$, there is $y\in H$ and a line $m$ incident with $1$ such that $\ell'=m^{(y,1)}$.
  \item[(2)]We have $\cP_1(g)=N_T(\bar{\ell})\setminus H$, and $N_T(\bar{\ell})=H\bar{\ell}$.
\end{itemize}
\end{lemma}
\begin{proof}
(1) Take a line $\ell'$ fixed by $g$. Since $g$ fixes two collinear points $1,u$, the fixed substructure $\cS_g$ is one of the cases (b)-(e) in Theorem \ref{thm:C_T(u)}. In particular, there is a fixed point $y$ on $\ell'$, i.e., $y\in H$. The line $m=\ell'^{(y^{-1},1)}$ is incident with $1$, and we can check that it is fixed by $g$. We have $\ell'=m^{(y,1)}$, and this proves the claim.

(2) Take a point $y$ in $\cP_1(g)$, so that $y\sim u^{-1}yu$. By applying $(y^{-1},u^{-1})\in G$ to both points, we see that $u$ and $yuy^{-1}$ are collinear. They are both collinear with $1$, so $yuy^{-1}\in\bar{\ell}$. We deduce that $(y,y)$ stabilizes the line $\ell$, i.e., $y\bar{\ell}y^{-1}=\bar{\ell}$, or equivalently $y\in N_T(\bar{\ell})$. We thus have $\cP_1(g)\subseteq N_T(\bar{\ell})\setminus H$. By reversing the above arguments, we deduce that $N_T(\bar{\ell})\setminus H\subseteq \cP_1(g)$, so equality holds. For $x\in H$, $(x,x)$ fixes both $1$ and $u$, so it fixes the line $\ell$. It follows that $x^{-1}\bar{\ell}x=\bar{\ell}$, i.e., $x\in N_T(\bar{\ell})$. We thus have $H\le N_T(\bar{\ell})$. It remains to show that $N_T(\bar{\ell})=H\bar{\ell}$. Take $x\in N_T(\bar{\ell})$. Then $xux^{-1}=u'$ for some $u'\in \bar{\ell}$.  The line $\ell^{(x,1)}$ is stabilized by $g$, since
\[
  \ell^{(x,1)(u,u)}=\ell^{(u'x,u)}=\ell^{(u',u)(x,1)}=\ell^{(x,1)}.
\]
Here we used the fact that $\bar{\ell}$ is a group, cf. (B1). By (1) there is $y\in H$ and a line $m$ incident with $1$ such that $\ell^{(x,1)}=m^{(y,1)}$. By comparing their point sets, we deduce that $\bar{\ell}x=\bar{m}y$. Since both $\bar{\ell}$ and $\bar{m}$ are subgroups by (B2), we have $\ell=m$. It follows that $x\in \bar{\ell}y\subseteq \bar{\ell}H$. This completes the proof.
\end{proof}

\begin{lemma}\label{lem_FixHSgrid}
The fixed substructure $\cS_g$ is not of type (a), (b), (d) or (e) in Theorem \ref{thm:C_T(u)}.
\end{lemma}
\begin{proof}
We write $X=\{g^ug^{-1}:g\in G\}$, and let $H=C_T(u)$. Since $g$ fixes two collinear points $1$ and $u$,  $\cS_g$ is not a partial ovoid, i.e., not of type (a).

Suppose that $\cS_g$ is of type (b), so that $\ell$ is the unique fixed line by $g$. We have $H\le \bar{\ell}$, and so $N_T(\bar{\ell})=\bar{\ell}$ by Lemma \ref{lem_FixgHS} (2). The stabilizer of $\ell$ in $D$ is $\{(x,x):x\in N_T(\ell)\}$ which has size $|\bar{\ell}|=s+1$, so its $D$-orbit has size $\frac{|T|}{s+1}=1+st$. On the other hand, each line in this $D$-orbit is incident with $1$, so $t+1\ge 1+st$: a contradiction.

Suppose that $\cS_g$ is a subquadrangle of order $(s',t')$ with $s'\ge 1$ and $t'\ge 2$. There are $c=t'+1$ fixed lines incident with a point of $H$,
and a fixed line contains $s'+1$ points of $H$. Moreover, we have
$|\cL_0(\theta)|=(1+t')(1+s't')$ and $|H|=(1+s')(1+s't')$. It follows from Lemma \ref{lem_FixgHS} (2) that
\[
  |N_T(\bar{\ell})|=\frac{|\bar{\ell}|\cdot|H|}{|\bar{\ell}\cap H|}=(1+s)(1+s't').
\]
By Lemma \ref{lem_FixgHS}, we have
\[
  |\cP_1(g)|=|N_T(\bar{\ell})|-|H|=(s-s')(1+s't').
\]
By Proposition \ref{prop_FixMult} (1), we have $|X\cap\Delta|=\frac{s-s'}{s'+1}$.
By Proposition \ref{prop_FixMult2} (2), we have
\[
   |\cL_0(g)|=(1+t')(1+s't')=\frac{(1+s')(1+s't')}{1+s}\left(t'+1+\frac{s-s'}{s'+1}\right).
\]
It yields $(s-s')t=0$ after simplification, so $s=s'$. Take a line $\ell'$ distinct from $\ell$ in $\cS_g$. Then $\bar{\ell}$ and $\bar{\ell'}$ are both  subgroups of $H$, and $\bar{\ell}\cap\bar{\ell'}=\{1\}$. The line sets $\{\ell^{(x,1)}:x\in\bar{\ell'}\}$ and $\{\ell'^{(1,y)}:y\in\bar{\ell}\}$ consist of pairwise disjoint lines, and  $\ell^{(x,1)}$, $\ell'^{(1,y)}$ are concurrent at the point $y^{-1}x$ for $x\in \bar{\ell'}$ and $y\in\bar{\ell}$. They form a subquadrangle of order $(s,1)$ fully contained in $\cS_g$. By \cite[2.2.2 (vi)]{Payne}, we deduce that $t=s^2$ and $t'=s$. This contradicts the bound $t\le s^2-s$ in Lemma \ref{lem_typeHS}. Hence $\cS_g$ is not of type (d) or (e).  This completes the proof.
\end{proof}

\begin{lemma}\label{lem_FixHSgrid2}
The substructure $\cS_g$ is  a grid with parameter $(s_1,s)$, where $s_1=|C_T(u)\cap \bar{\ell}|-1$. In particular, we have $|C_T(u)|=(s+1)(|C_T(u)\cap \bar{\ell}|+1)$.
\end{lemma}
\begin{proof}
Let $X=\{g^ug^{-1}:g\in G\}$, and let $H=C_T(u)$. By Lemma \ref{lem_FixHSgrid}, $\cS_g$ is a grid with parameter $(s_1,s_2)$. In particular, there is exactly one more fixed line $m$ that is incident with $1$. We assume without loss of generality that $s_1+1=|\bar{\ell}\cap H|$, $s_2+1=|\bar{m}\cap H|$. There are $c=2$ fixed lines incident with a fixed point, and $|H|=(s_1+1)(s_2+1)$, $|\cL_0(g)|=s_1+s_2+2$. By Lemma \ref{lem_FixgHS} (2), we deduce that
\[
 |N_T(\bar{\ell})|=\frac{|H|\cdot|\bar{\ell}|}{H \cap\bar{\ell}}=(s+1)(s_2+1),
\]
and similarly $|N_T(\bar{m})|=(s+1)(s_1+1)$. By the same lemma, we have
\[
  |\cP_1(g)|=|N_T(\bar{\ell})|-|H|=(s-s_1)(s_2+1).
\]
By Proposition \ref{prop_FixMult} (1), we have $|X\cap\Delta|=\frac{|\cP_1(g)|}{|H|}=\frac{s-s_1}{s_1+1}$. By Proposition 3.2 (2), we have $(|X\cap\Delta|+2)|H|=(s+1)|\cL_0(g)|$. After simplification, we obtain $(s-s_2)(s_1+s_2+2)=0$, so $s=s_2$ as desired. This completes the proof.
\end{proof}

We are now ready to complete the proof of Theorem \ref{thm_HSnonext}. The group $D=\{(x,x):x\in T\}$ permutes the set of lines incident with $1$. We write $\ell_1,\ldots,\ell_b$ for a complete set of $D$-orbit representatives for this action. By (B2), $\overline{\ell_i}$ is a subgroup of order $s+1$ and it contains an involution $u_i$. By Lemma \ref{lem_FixgHS} (2), we have $N_T(\overline{\ell_i})=\overline{\ell_i} C_T(u_i)$. By Lemma \ref{lem_FixHSgrid2}, we have $|C_T(u_i)|=(s+1)|C_T(u_i)\cap \overline{\ell_i}|$. It follows that
\begin{equation}\label{eqn_NTli}
  |N_T(\overline{\ell_i})|=\frac{(s+1)|C_T(u_i)|}{|C_T(u_i)\cap \overline{\ell_i}|}=(s+1)^2.
\end{equation}
The stabilizer of a line $\ell$ in $D$ is $\{(x,x):x\in N_T(\bar{\ell})\}$, so $t+1=\sum_{i=1}^{b}\frac{|T|}{|\N_{T}(\overline{\ell_i})|}$.  By using \eqref{eqn_NTli} and the fact $|T|=(1+s)(1+st)$, we deduce that $t+1=\frac{1+st}{s+1}b$. It follows that $s+t=(b-1)(1+st)$. On the other hand, $s+t$ divides $1+st$ by Lemma \ref{lem_typeHS}. We deduce that $b=2$, and $s+t=1+st$, i.e., $(s-1)(t-1)=0$: a contradiction. This completes the proof of Theorem \ref{thm_HSnonext}.

\section{Point-primitive automorphism groups of type SD or CD}
In this section, we suppose that $\mathcal{S}=(\mathcal{P},\mathcal{L})$ is a finite  thick generalized quadrangle of order $(s,t)$ admitting a point-primitive automorphism group $G$  of type SD or CD. There is a finite nonabelian simple group $T$ such that $\textup{soc}(G)=(T^k)^r$, where $k\ge 2$ and $r\ge 1$. Let
\[
 D=\{(t,\cdots,t):t\in T\}\le T^k,\quad  M:=(T^{k-1}\times\{1\})^r,
\]
so that $D^r$ normalizes $M$ and $\textup{soc}(G)=D^r M$. There is a point $p\in\cP$ whose stabilizer in $\textup{soc}(G)$ is $D^r$, and the group $M$ acts regularly on the point set $\cP$.   We aim to show that the conditions in Table~\ref{tab:1} hold and thus complete the proof of Theorem \ref{thm_PtPrim}. 

\begin{lemma}\label{lem_stsmall}
We have $\min\{s,t\}\ge 4$.
\end{lemma}
\begin{proof}
We have $|\cP|=|M|=|T|^{(k-1)r}\ge 60$, so $\min\{s,t\}\ge 3$. If $\min\{s,t\}=3$, then $\cS$ is Payne derived quadrangle of $W(4)$, cf. Remark \ref{rem_smallpara}. Its Singer group has order $2^6$, and so is solvable. This contradicts the fact that $M$ is nonsolvable, so $\min\{s,t\}\ge 4$.
\end{proof}

\begin{lemma}\label{lem:SD_2/3}
For any $x\in T\setminus\{1\}$, we have $|\C_T(x)|<|T|^{1-r/4}$. Moreover, we have $r\le 3$.
\end{lemma}
\begin{proof}
Let $a=(x,\ldots,x)\in D$ and $\theta=(a,1,\ldots,1)\in D^r$. Then $\theta$ is a multiplier of $\cS$, and we have $C_G(\theta)=\C_T(x)^{k-1}\times (T^{k-1})^{r-1}$. On the other hand, we have $|C_G(\theta)|<|T|^{3(k-1)r/4}$ by Lemma \ref{lem_stsmall} and Corollary \ref{cor:2/3}. The claim $|\C_T(x)|<|T|^{1-r/4}$ follows after simplification. Since $|\C_T(x)|>1$, we have $1-r/4>0$, i.e., $r\le 3$. This completes the proof.
\end{proof}

In the next few lemmas we adopt the results on centralizer orders in nonabelian finite simple groups in \cite{Bamberg 2019} for our purpose.

\begin{lemma}\label{lem:spora}
If $T$ is a sporadic simple group or the Tits group ${}^2F_4(2)'$, then  there is some $x\in T\setminus\{1\}$ such that $|\C_T(x)|>|T|^{1/4}$. Moreover, if $T$ is not one of $\rm{M}_{11}$, $\rm{M}_{12}$, $\rm{M}_{22}$, $\rm{M}_{23}$, Th, ${\rm J}_1$, O'N, $\rm{J}_3$, Ru or Ly, then there is $x\in T\setminus\{1\}$ such that
$|\C_T(x)|>|T|^{1/2}$.
\end{lemma}
\begin{proof}
The results follow by checking maximal centralizer orders in the ATLAS \cite{ATLAS}.
\end{proof}

\begin{lemma}\label{lem:alt}
If $T\cong {\rm Alt}_n$ with $n\ge 5$, then
\begin{itemize}
\item[(1)] there is some $x\in T\setminus\{1\}$ such that $|\C_T(x)|>|T|^{3/4}$ if $n\ge 16$;
\item[(2)] there is some $x\in T\setminus\{1\}$ such that $|\C_T(x)|>|T|^{1/2}$ if $n\ge 8$;
\item[(3)] there is some $x\in T\setminus\{1\}$ such that $|\C_T(x)|>|T|^{1/4}$.
\end{itemize}
\end{lemma}
\begin{proof}
Take $x\in T$ to be a $3$-cycle, so that $|\C_T(x)|=\frac{3}{2}(n-3)!$. It holds that $|\C_T(x)|>|T|^{3/4}$ if $n\ge 16$, $|\C_T(x)|>|T|^{1/2}$ if $n\ge 8$ and $|\C_T(x)|>|T|^{1/4}$ if $n\ge 5$. This completes the proof.
\end{proof}

\begin{lemma}\label{lem:clssim}
Suppose that $T$ is a finite simple classical group.
\begin{itemize}
\item [(1)] There is some $x\in T\setminus\{1\}$ such that $|\C_T(x)|\ge|T|^{3/4}$ if $T$ is one of the following groups: $\PSL(n,q)$ with $n\ge 8$, ${\rm PSU}(n,q)$ with $n\ge 7$, ${\rm PSp}(2n,q)$ with $n\ge 4$, $\Omega(2n+1,q)$ with $n\ge 4$, or ${\rm P\Omega}(2n,q)$ with $n\ge 9$.
\item [(2)] There is some $x\in T\setminus\{1\}$ such that $|\C_T(x)|\ge|T|^{1/2}$ if $T$ is one of the following groups: $T=\PSL(n,q)$ with $n\ge 4$, ${\rm PSU}(n,q)$ with $n\ge 4$, ${\rm PSp}(2n,q)$ with $n\ge 2$, $\Omega(2n+1,q)$ with $n\ge 3$, or ${\rm P\Omega}^\epsilon(2n,q)$ with $n\ge 4$.
\item [(3)] There is some $x\in T\setminus\{1\}$ such that $|\C_T(x)|\ge|T|^{1/4}$.
\end{itemize}
\end{lemma}
\begin{proof}
We take the same elements $x$'s as in the proof of \cite[Lemma 3.4]{Bamberg 2019} and verify that $|C_T(x)|$ satisfy the respective bounds here.  Let $q=p^f$ with $p$ prime and $f\ge 1$ in each case. We refer to the type of a conjugacy class of elements of prime order as in \cite{Tim}.

(a) Assume that $T=\PSL(n,q)$ with $n\ge 2$, and set $d=(n,q-1)$. Let $x\in T$ be an element of order $p$ with one Jordan block of size $2$ and $n-1$ Jordan blocks of size $1$ as in \cite[Table B.3, Row~1]{Tim}. We have
\[
 |\C_T(x)|=\frac{1}{d}q^{n(n-1)/2}\prod_{i=1}^{n-2}(q^i-1), \textup{ and } |T|=\frac{1}{d}q^{n(n-1)/2}\prod_{i=2}^{n}(q^i-1).
\]
It is straightforward to verify that $|C_T(x)|>|T|^{3/4}$ for $n\ge 8$, $|C_T(x)|>|T|^{1/2}$ for $n\ge 4$, and $|C_T(x)|>|T|^{1/4}$ for $n\ge 2$.

(b) Assume that $T={\rm PSU}(n,q)$ with $n\ge 3$, and set $d=(n,q+1)$. Let $x\in T$ be an element of order $p$ with one Jordan block of size $2$ and $n-1$ Jordan blocks of size $1$  as in \cite[Table B.4, Row~1]{Tim}.
Then, 
\[
|\C_T(x)|=\frac{1}{d}q^{n(n-1)/2}\prod_{i=1}^{n-2}(q^i-(-1)^i), \, \textup{ and } |T|=\frac{1}{d}q^{n(n-1)/2}\prod_{i=2}^{n}(q^i-(-1)^i).
\]
It is straightforward to verify that $|C_T(x)|>|T|^{3/4}$ for $n\ge7$, $|C_T(x)|>|T|^{1/2}$ for $n\ge4$, and $|C_T(x)|>|T|^{1/4}$ for $n\ge 3$.
 
(c) Assume that $T={\rm PSp}(2n,q)$ with $n\ge 2$, and set $d=(2,q-1)$. If $p>2$, take $x\in T$ to be an element of order $p$ with one Jordan block of size $2$ and $2(n-1)$ Jordan blocks of size $1$. If $p=2$, take $x$ to be an involution of type $b_1$ as in \cite[Table 3.4.1]{Tim}. We have  
\[
  |\C_T(x)|=\frac{1}{d}q^{n^2}\prod_{i=1}^{n-1}(q^{2i}-1), \, \textup{ and } |T|=\frac{1}{d}q^{n^2}\prod_{i=1}^{n}(q^{2i}-1).
\]
We can verify that $|\C_T(x)|>|T|^{1/2}$ always holds, and $|\C_T(x)|>|T|^{3/4}$ if $n\ge4$.
 
(d) Assume that $T={\rm \Omega}(2n+1,q)$ with $q$ odd and $n\ge 3$. Take
$x\in T$ to be an involution of type $t_n$ or $t_n'$  respectively as in \cite[Sections 3.5.2.1 and 3.5.2.2]{Tim}, according as $q\equiv 1\pmod 4$ or $q\equiv 3\pmod 4$.
Then  we have
\[
  |\C_T(x)|=q^{n^2-n}(q^n\mp 1)\prod_{i=1}^{n-1}(q^{2 i}-1), \, \textup{ and } |T|=\frac{1}{2}q^{n^2}\prod_{i=1}^n(q^{2i}-1).
\]
We can verify that $|\C_T(x)|>|T|^{1/2}$ always holds, and $|\C_T(x)|>|T|^{3/4}$ if $n\ge 4$.

(e) Assume that $T={\rm P\Omega}^{\epsilon}(2n,q)$ with  $n\ge 4$, and set $d=(4,q^n-\epsilon)$. Suppose that $p>2$ first. Take $x\in T$ to be an element  of order $p$ with one Jordan block of size $2(n-2)$ and two Jordan blocks of size $2$ (see \cite[Section 3.5.3]{Tim}). We have
\[
|\C_T(x)|\ge\frac{1}{8}q^{n^2-2}(q^2- 1)\prod_{i=1}^{n-3}(q^{2 i}-1), \, \textup{ and } |T|=\frac{1}{d}q^{n(n-1)}(q^n-\epsilon)\prod_{i=1}^{n-1}(q^{2i}-1).
\]
We can verify that $|\C_T(x)|>|T|^{1/2}$ always holds, and $|\C_T(x)|>|T|^{3/4}$ if $n\ge 9$. Suppose that $p=2$. Take $x\in T$ to be an involution of type $a_2$ as in \cite[Table 3.5.1]{Tim}. Then,
\[
|\C_T(x)|\ge \frac{1}{4}q^{n^2-2}(q^2- 1)\prod_{i=1}^{n-3}(q^{2 i}-1), \, \textup{ and } |T|=q^{n(n-1)}(q^n-\epsilon)\prod_{i=1}^{n-1}(q^{2i}-1).
\]
We  also have $|\C_T(x)|>|T|^{1/2}$  for $n\ge 4$, and $|\C_T(x)|>|T|^{3/4}$ if $n\ge 9$.  This completes the proof.
\end{proof}

\begin{lemma}\label{lem:excep}
If $T$ is a finite simple group of exceptional Lie type, then
\begin{enumerate}[(1)]
\item there exists $x\in T\setminus\{1\}$ such that $|\C_T(x)|>|T|^{1/4}$;
\item there exists $x\in T\setminus\{1\}$ such that $|\C_T(x)|>|T|^{1/2}$ if $T$ is not of type ${}^2{\rm B}_2$ or ${}^2{\rm G}_2$;
\item there exists $x\in T\setminus\{1\}$ such that $|\C_T(x)|>|T|^{3/4}$ if $T$ has type ${\rm E}_8$.
\end{enumerate}
\end{lemma}
\begin{proof}
(1) By \cite[Lemma 3.3 (ii)]{Bamberg 2019}, there is some  $x\in T\setminus\{1\}$ such that  $|\C_T(x)|\ge |T|^{2/5}$  if and only if $T$ does not have type ${}^2{\rm B}_2$. If $T\cong {}^2{\rm B}_2(q)$ with $q=2^{2m+1}$ for some $m\ge 1$, then  it has a unique conjugacy class of involutions, and for an involution $x$ in $T$, we have  $|C_T(x)|=q^2$ by \cite{Suzuki}. It is readily checked that $|C_T(x)|=q^2>|T|^{1/4}$ in this case.

(2) If $T$ has type ${\rm E}_7$, $\rm{E}_6$, ${}^2E_6$, ${\rm F}_4$ or ${}^3{\rm D}_4$, there exists $x\in T\setminus\{1\}$ such that $|\C_T(x)|>|T|^{3/5}$ by \cite[Lemma 3.3 (i)]{Bamberg 2019}. If $T$ has type ${}^2{\rm F}_4$ or ${\rm G}_2$, take $x$  to be a unipotent element of type $(\tilde{A_1})_2$, $A_1$ in \cite[Tables 22.2.5, 22.2.6]{Liebeck}, so that $|\C_T(x)|=q^{10}|{}^2{\rm B}_2(q)$,  $q^5|\PSL(2,q)|$. We then check that $|\C_T(x)|> |T|^{1/2}$ holds in both cases.

(3) If $T\cong {\rm E}_8(q)$, we take $x\in T$ to be a uniponent element of type $A_1$ as in \cite[Table 22.2.1]{Liebeck}. Then we have
\[
|\C_T(x)|\ge\frac{1}{2}q^{120}\prod_{i\in\{2,6,8,10,12,14,18\}}(q^i-1),\,\textup{ and } |T|=q^{120}\prod_{i\in\{2,8,12,14,18,20,24,30\}}(q^i-1).
\]
One can readily check that $|\C_T(x)|>|T|^{3/4}$. This completes the proof.
\end{proof}

\noindent\textbf{Proof of Theorem \ref{thm_PtPrim}:} Let $\cS$ be a finite thick generalized quadrangle, and suppose that it has a  point-primitive automorphism group $G$. The action of $G$ on the point set does not have O'Nan-Scott  type HC by \cite[Theorem 1.1]{Bamberg 2019}.  We have shown that it does not have O'Nan-Scott type HS in Theorem \ref{thm_HSnonext}. Suppose that $G$ has type SD or CD, and $\textup{soc}(G)=(T^k)^r$, where $T$ is a finite nonabelian simple group, $k\ge 2$ and $r\ge1$. First assume that $r=1$, so that $G$ has type SD. The same arguments as in the proof of Theorem \ref{thm_HSnonext} exclude the case $k=2$. If $k\ge 3$, we have  $|\C_T(x)|<|T|^{3/4}$ for any $x\in T\setminus\{1\}$ by Lemma \ref{lem:SD_2/3}. We then use the results in Lemmas \ref{lem:spora}-\ref{lem:excep} to obtain the list of potential $T$'s in Table \ref{tab:1}. If $r>1$ and $G$ has type CD, then $r\le 3$ and $|\C_T(x)|<|T|^{1-r/4}$ for any $x\in T\setminus\{1\}$ by Lemma \ref{lem:SD_2/3}. We obtain the list of potential $T$'s in Table \ref{tab:1} similarly by using Lemmas \ref{lem:spora}-\ref{lem:excep}. This completes the proof.\eproof
\vspace*{10pt}

\section{Concluding remarks}
In this paper, we introduce the notion of multipliers for a Singer quadrangle and study their basic properties. As an application, we show that a point-primitive automorphism group of a thick generalized quadrangle cannot have O'Nan-Scott type HS (holomorph simple) which answers an open problem in \cite{Bamberg 2019}. We also exclude some potential point-primitive groups of type SD (simple diagonal) and CD (compound diagonal) listed in \cite[Theorem 1.1]{Bamberg 2019}. Therefore, the notion of multipliers provides a new perspective in the study of highly transitive generalized quadrangles with a point-regular automorphism group. There is a rich theory on the multiplier group of projective planes with a Singer group, and it is desirable to develop an analogous theory for Singer quadrangles and find new applications.\\

\noindent\textbf{Acknowledgements.} This research was supported by National Key R$\&$D Program of China under grant number 2025YFA1017700, NSFC grants No. 12225110 and 12461061.

\begin{center}
	\scriptsize
		\linespread{0.5}

\begin{thebibliography}{10}




\bibitem{Bamberg2021}
J. Bamberg, J.~P. Evans, 
\newblock No sporadic almost simple group acts primitively on the points of a generalised quadrangle, 
\newblock{\em Discrete Math.} {\bf 344} (2021), no.~4, Paper No. 112291.

\bibitem{Bamberg2011}
J. Bamberg, M. Giudici,
\newblock Point regular groups of automorphisms of generalised quadrangles,
\newblock {\em J. Combin. Theory Ser. A} {\bf 118} (2011), no.~3, 1114--1128.





\bibitem{Bamberg2012}
J. Bamberg, M. Giudici, J. Morris, G. F. Royle, P. Spiga,
\newblock Generalised quadrangles with a group of automorphisms acting primitively on points and lines,
\newblock {\em J.Combin. Theory Ser. A}  {\bf 119} (2012), no.~7, 1479--1499.



\bibitem{Bamberg2016}
J. Bamberg, S. P. Glasby, T. Popiel, C. E. Praeger,
\newblock Generalized quadrangles and transitive pseudo-hyperovals,
\newblock {\em J. Combin. Des.} {\bf 24} (2016), no.~4, 151--164.



\bibitem{Bamberg Popiel2017}
J. Bamberg, T. Popiel, C. E. Praeger,
\newblock Point-primitive, line-transitive generalised quadrangles of holomorph type,
\newblock {\em J. Group Theory} {\bf 20} (2017), no.~2, 269--287.



\bibitem{Bamberg 2019}
J. Bamberg, T. Popiel, C. E. Praeger,
\newblock Simple groups, product actions, and generalized quadrangles,
\newblock{\em Nagoya Math. J.} {\bf 234} (2019), 87--126.


\bibitem{Burness}
 T. C. Burness, E. Covato,
 \newblock On the involution fixity of simple groups,
 \newblock{\em Proc. Edinb. Math. Soc.} {\bf 64} (2021), no.~2, 408--426.



\bibitem{ATLAS}
J. H. Conway, R. T. Curtis, S. P. Norton, R. A. Parker, R. A. Wilson,
\newblock  ATLAS of finite groups,
\newblock{\em Oxford University Press,} Oxford, 1985.





\bibitem{Hall47}M. Hall Jr., 
\newblock Cyclic projective planes,
\newblock{\em Duke Math. J.} {\bf 14} (1947), 1079--1090.


\bibitem{Tim}
T.~C. Burness, M. Giudici, 
\newblock {\it Classical groups, Derangements and Primes}, \newblock Australian Mathematical Society Lecture Series, 25, Cambridge Univ. Press, Cambridge, 2016.






\bibitem{Feng}
T. Feng,
\newblock
On finite generalized quadrangles of even order,
\newblock {\em Adv. Math.} {\bf 429} (2023), Paper No. 109181.


\bibitem{FengLi}
T. Feng, W. Li,
\newblock The point regular automorphism groups of the Payne derived quadrangle of $W(q)$,
\newblock{\em J. Combin. Theory Ser. A} {\bf 179} (2021), Paper No. 105384.

\bibitem{Ho89}
C.~Y. Ho,
\newblock{On multiplier groups of finite cyclic planes,}
\newblock{\em J. Algebra} {\bf 122} (1989), no.~1, 250--259.

\bibitem{Ho93}
C.~Y. Ho,
\newblock{Planar Singer groups and groups of multipliers,}
\newblock{\em S\=urikaisekikenky\=usho K\ Boky\=uroku} no. 840 (1993), 65--69.





\bibitem{HoPAMS}
C.~Y. Ho, 
\newblock{Singer groups, an approach from a group of multipliers of even order,}
\newblock{\em Proc. Amer. Math. Soc.} {\bf 119} (1993), no.~3, 925--930.


\bibitem{Ho93basic}
C.~Y. Ho, 
\newblock{Some basic properties of planar Singer groups,}
\newblock{\em Geom. Dedicata} {\bf 55} (1995), no.~1, 59--70.

\bibitem{Lu}
T. Feng, J. Lu,
 \newblock{ On finite generalized quadrangles with  $\PSL(2,q)$  as an automorphism group},
 \newblock{\em Des. Codes Cryptogr.}  {\bf 91} (2023), no.~6, 2347--2364.

\bibitem{Ghinelli}
D. Ghinelli,
\newblock Regular groups on generalized quadrangles and nonabelian difference sets with multiplier $-1$,
\newblock {\em Geom. Dedicata} {\bf 41} (1992), no.~2, 165--174.


\bibitem{Ott}U. Ott, 
\newblock{On generalized quadrangles with a group of automorphisms acting regularly on the point set, difference sets with $-1$ as multiplier and a conjecture of Ghinelli}, \newblock{\em European J. Combin.} {\bf 110} (2023), Paper No. 103681.


\bibitem{Liebeck}
M. W. Liebeck, G. M. Seitz,
\newblock {\it Unipotent and Nilpotent Classes in Simple Algebraic
Groups and Lie Algebras},
\newblock Mathematical Surveys and Monographs 180, American
Mathematical Society, Providence, RI, 2012.



\bibitem{Zou}
J. Lu, Y. Zhang, H. Zou,
\newblock{Nonexistence of generalized quadrangles admitting a point-primitive and line-primitive automorphism group with socle $\rm{PSU}(3,q)$, $q\geq 3$},
\newblock{\em J. Algebraic Combin. }
 {\bf 60} (2024), no.~3, 871--898.













\bibitem{Liebeck}
 M. W. Liebeck, A. Shalev,
 \newblock On fixed points of elements in primitive permutation groups,
 \newblock{\em J. Algebra}  {\bf 421} (2015), 438--459.






\bibitem{Payne}
S. E. Payne, J. A. Thas,
\newblock {\em Finite Generalized Quadrangles},
\newblock second ed., EMS Ser. Lect. Math., European Mathematical Society (EMS), Z{\"u}rich, 2009.

\bibitem{Praeger}
 C. E. Praeger,
 \newblock Finite quasiprimitive graphs, in: Surveys in Combinatorics,
 \newblock{ \em University of Western Australia. Department of Mathematics} (1996) 65--85.

\bibitem{Suzuki}
M. Suzuki, 
\newblock On a class of doubly transitive groups, 
\newblock{\em Ann. of Math. } (2) {\bf 75} (1962), 105--145.


\bibitem{Swartz}
E. Swartz,
\newblock On generalized quadrangles with a point regular group of automorphisms,
\newblock {\em European J. Combin.} {\bf 79} (2019), 60--74.


\bibitem{Winter&Thas}
S. De~Winter and K. Thas,
\newblock Generalized quadrangles with an abelian Singer group,
\newblock {\em Des. Codes Cryptogr. }{\bf 39} (2006), no.~1, 81--87.


\bibitem{Winter&Thas2}
S. De~Winter and K. Thas,
\newblock Generalized quadrangles admitting a sharply transitive Heisenberg group,
\newblock {\em Des. Codes Cryptogr.} {\bf 47} (2008), no.~1-3, 237--242.


\bibitem{WTSSinger} S. De~Winter, K. Thas, E.E. Shult, Singer Quadrangles, Oberwolfach Preprint OWP, 2009-07.

\bibitem{Yoshiara}
S. Yoshiara,
A generalized quadrangle with an automorphism group acting regularly on the points,
{\em European J. Combin.} {\bf 28} (2007), no.~2, 653--664.





\end{thebibliography}
	\bibliographystyle{plain}

\end{center}

\end{document}